\documentclass[11pt]{amsart}

\usepackage[a4paper,hmargin=3.5cm,vmargin=3.5cm]{geometry}
\usepackage{amsfonts,amssymb,amscd,amstext}
\usepackage{graphicx}
\usepackage[dvips]{epsfig}
\usepackage{mathpazo}
\usepackage{enumerate}
\usepackage{titlesec}

\def\Scal{\mathcal{S}}
\def\Rcal{\mathcal{R}}

\def\Ncal{\mathcal{N}}

\def\Acal{\mathcal{A}}

\def\Jcal{\mathcal{J}}

\def\C{\mathbb{C}}\def\c{\mathbb{C}}
\def\d{\mathbb{D}}
\def\R{\mathbb{R}}\def\r{\mathbb{R}}

\def\n{\mathbb{N}}
\def\s{\mathbb{S}}
\def\h{\mathbb{H}}
\def\M{\mathbb{M}}

\def\sgot{\mathfrak{s}}

\def\ngot{\mathfrak{n}}
\def\mgot{\mathfrak{m}}
\def\Agot{\mathfrak{A}}
\def\Mgot{\mathfrak{M}}
\def\Ggot{\mathfrak{G}}
\def\Ogot{\mathfrak{O}}

\titleformat{\section}%[display]
{\bfseries\large} {\thesection{.}}{0.2cm}{}
\titleformat{\subsection}%[display]
{\bfseries} {\thesubsection{.}}{0.15cm}{}
\titleformat{\subsubsection}%[display]
{\em}{\thesubsubsection{.}}{0.15cm}{}
\newtheorem{theorem}{Theorem}[section]
\newtheorem{proposition}[theorem]{Proposition}
\newtheorem{claim}[theorem]{Claim}

\newtheorem{corollary}[theorem]{Corollary}
\newtheorem{remark}[theorem]{Remark}

\theoremstyle{definition}
\numberwithin{equation}{section}
\numberwithin{figure}{section}

\usepackage{color}

\begin{document}

\thispagestyle{empty}

%% Title
\vspace*{1cm}
\noindent{\bf\LARGE Harmonic  diffeomorphisms between domains in the Euclidean $2$-sphere}

\vspace*{0.5cm}

%% Authors
\noindent{\bf\large Antonio Alarc\'{o}n {\small $\;\bullet\;$} Rabah Souam}

%% Addresses and finantial support
\footnote[0]{\vspace*{-0.4cm}

\noindent A. Alarc\'{o}n

\noindent Departamento de Geometr\'{\i}a y Topolog\'{\i}a, Universidad de Granada, E-18071 Granada, Spain

\noindent e-mail: {\tt alarcon@ugr.es}

\vspace*{0.1cm}

\noindent R. Souam

\noindent Institut de Math\'{e}matiques de Jussieu, CNRS UMR 7586, Universit\'{e} Paris Diderot – Paris 7, ``G\'{e}om\'{e}trie et Dynamique'', Site Chevaleret, Case 7012, 75205 – Paris Cedex 13,
France

\noindent e-mail: {\tt souam@math.jussieu.fr}

\vspace*{0.1cm}

\noindent The first author is partially supported by MCYT-FEDER research project MTM2007-61775 and Junta de Andaluc\'{i}a Grant P09-FQM-5088}

%% Abstract, keywords and MSC
\vspace*{1cm}

{\small
\noindent {\bf Abstract}\hspace*{0.1cm} We study the existence or not of harmonic diffeomorphisms between certain domains in the Euclidean $2$-sphere. In particular, we show harmonic diffeomorphisms from circular domains in the complex plane onto finitely punctured spheres, with at least two punctures. This result follows from a general existence theorem for maximal graphs in the Lorentzian product $\M\times\r_1,$ where $\M$ is an arbitrary $\ngot$-dimensional compact Riemannian manifold, $\ngot\geq 2.$ In contrast, we show that there is no harmonic diffeomorphism from the unit complex disc onto the once punctured sphere and no harmonic diffeomeorphisms from finitely punctured spheres onto circular domains in the Euclidean $2$-sphere.

\vspace*{0.1cm}

\noindent{\bf Keywords}\hspace*{0.1cm} Harmonic diffeomorphisms $\cdot$ Maximal graphs

\vspace*{0.1cm}

\noindent{\bf Mathematics Subject Classification (2010)}\hspace*{0.1cm} 53C43 $\cdot$ 53C42 $\cdot$ 53C50
}

%%%%%%
%%%%%%
%%%%%%

\section{Introduction}\label{sec:intro}

In 1952, Heinz \cite{He} proved there is no harmonic diffeomorphism from the unit complex disk $\d$ onto the complex plane $\c,$ with the euclidean metric. Later, Schoen and Yau \cite{SY} asked whether Riemannian surfaces which are related by a harmonic diffeomorphism are quasi-conformally related. Recently, Collin and Rosenberg \cite{CR} showed a harmonic diffeomorphism from $\c$ onto the hyperbolic plane $\h^2,$ disproving a conjecture by Schoen and Yau \cite{SY}. To do that, they constructed an entire minimal graph $\Sigma$ over $\h^2$ in the Riemannian product $\h^2\times\r,$ with the conformal type of $\c.$ Then the vertical projection $\Sigma\to\h^2$ is a surjective harmonic diffeomorphism. See \cite{GR} for further generalizations.

%%%%%%%%%
%%%%%%%%%
%%%%%%%%%

Let  $\s^2$ and $\overline{\c}$ denote  the $2$-dimensional Euclidean unit sphere and the Riemann sphere, respectively. A domain in $\overline{\C}$ is said to be a circular domain if every connected component of its boundary is a circle.

In this paper we study the existence or not of harmonic diffeomorphisms between certain domains in $\s^2.$ 
Our main result asserts:

\begin{quote}
{\bf Theorem I.}
{\em %The following assertions hold:

\begin{enumerate}[{\rm (i)}]
\item For any $\mgot\in\n,$ $\mgot\geq 2,$ and any subset $\{p_1,\ldots,p_\mgot\}\subset\s^2$ there exist a circular domain $U\subset\overline{\c}$ and a harmonic diffeomorphism $\phi:U\to\s^2-\{p_1,\ldots,p_\mgot\}.$

\item There exists no harmonic diffeomorphism $\varphi:\d\to\s^2-\{p\}.$ 

\item For any $\mgot\in\n,$ any subset $\{z_1,\ldots,z_\mgot\}\subset\overline{\c}$ and any pairwise disjoint closed discs $D_1,\ldots,D_\mgot$ in $\s^2$ there exists no harmonic diffeomorphism $\psi:\overline{\c}-\{z_1,\ldots,z_\mgot\}\to\s^2-\cup_{j=1}^\mgot D_j.$ 
\end{enumerate}
}
\end{quote}

%\begin{quote}
%{\bf Theorem I.}
%{\em Let $\M$ be a compact Riemannian surface.
%
%Then the following assertions hold:
%
%\begin{enumerate}[{\rm (i)}]
%\item For any $\mgot\in\n,$ $\mgot\geq 2,$ and any subset $\{p_1,\ldots,p_\mgot\}\subset\M$ there exist an open Riemann surface $\Ncal$ and a harmonic diffeomorphism $\phi:\Ncal\to\M-\{p_1,\ldots,p_\mgot\}$ such that every end of $\Ncal$ is of hyperbolic conformal type.
%
%In particular, if $\M$ is the 2-dimensional Euclidean sphere $\s^2$ then $\Ncal$ can be chosen a circle domain in the Riemann sphere $\overline{\C}$ such that every boundary component of $\overline{\C}$ is a circle.
%
%\item There exists no harmonic diffeomorphism $\phi:\d\to\s^2-\{p\},$ where $\d$ denotes the unit complex disc and $p\in\s^2.$
%
%\item For any parabolic planar domain $\Dcal\subset\c$ there exists no harmonic local diffeomorphism $\phi:\Dcal\to\s^2.$ 
%\end{enumerate}
%}
%\end{quote}

{ Notice that Theorem I is related to Schoen and Yau's questions, since circular domains are of hyperbolic conformal type whereas $\overline{\c}$ with a finite set removed is parabolic. Moreover, it is worth mentioning that Items (i) and (iii) actually follow from much more general results (see Corollary \ref{co:diffeo} and Proposition \ref{pro:converse}). Concretely, we show that 
given a compact Riemannian surface $\M$ and a subset $\{p_1,\ldots,p_\mgot\}\subset\M,$ $\mgot\geq 2,$ then there exist an open Riemann surface $\Rcal$ and a harmonic diffeomorphism $\phi:\Rcal\to\M-\{p_1,\ldots,p_\mgot\}$ such that every end of $\Rcal$ is of hyperbolic conformal type.
}

Our strategy to show the harmonic diffeomorphism of Item (i) in Theorem I consists of constructing a maximal graph $\Sigma$ over $\s^2-\{p_1,\ldots,p_\mgot\}$ in the Lorentzian manifold $\s^2\times\r_1,$ with the conformal type of a circular domain. Then, the projection $\Sigma\to\s^2-\{p_1,\ldots,p_\mgot\}$ is a surjective harmonic diffeomorphism. 

In this direction, we prove the following general existence result:

\begin{quote}
{\bf Theorem II.}
{\em Let $\M=(\M,\langle\cdot,\cdot\rangle_\M)$ be a compact Riemannian manifold without boundary of dimension $\ngot\in\n,$ $\ngot\geq 2,$ and denote by  $\M\times\r_1$ the product manifold $\M\times\r$ endowed with the Lorentzian metric $\langle\cdot,\cdot\rangle_\M-dt^2.$ Let $\mgot\in\n,$ $\mgot\geq 2,$ and let $\Agot=\{(p_i,t_i)\}_{i=1,\ldots,\mgot}$ be a subset of $\M\times\r_1$ such that
\begin{itemize}
\item $p_i\neq p_j$ and
\item $|t_i-t_j|<{\rm dist}_\M(p_i,p_j),$ $\forall \{i,j\}\subset\{1,\ldots,\mgot\}$ with $i\neq j.$
\end{itemize}

Then there exists exactly one entire graph $\Sigma(\Agot)$ over $\M$ in $\M\times\r_1$ such that
\begin{itemize}
\item $\Agot\subset\Sigma(\Agot)$ and
\item $\Sigma(\Agot)-\Agot$ is a spacelike maximal graph over $\M-\{p_i\}_{i=1,\ldots,\mgot}.$ 
\end{itemize}

Moreover the space  $\Ggot_\mgot$  of entire maximal graphs over $\M$ in $\M\times\r_1$ with precisely $\mgot$ singularities, endowed with the topology of uniform convergence, is non-empty, and there exists a $m!$-sheeted covering, $\overline{\Ggot}_{\mgot}\to \Ggot_\mgot,$ where  
 $\overline{\Ggot}_{\mgot}$ is an open subset of $(\M\times\r)^\mgot.$ }

\end{quote}

Let us point out that our method is different from the one of Collin and Rosenberg \cite{CR} and strongly relies on the theory of maximal hypersurfaces in Lorentzian manifolds. More precisely, it is based on the construction of maximal hypersurfaces with isolated singularities in Lorentzian products $\M\times\r_1.$
The study of {\it complete} maximal surfaces, with a finite number of singularities and their moduli spaces in the 3-dimensional Minkowski space $\mathbb L^3,$  was developed by Fern\'{a}ndez, L\'{o}pez and Souam \cite{FLS}. Their study strongly relies on the Weierstrass representation for maximal surfaces in $\mathbb L^3$ . Our approach here relies on a different idea which consists of dealing with the existence problem in Theorem II
as a generalized Dirichlet problem. Let us also point out that Klyachin and Miklyukov \cite{KM}
have obtained results on the existence of solutions, with a finite number of singularities, to the maximal Hypersurface equation in the $n$-dimensional Minkowski space $\mathbb L^{n}$ with prescribed boundary conditions.

Harmonic maps from Riemann surfaces into $\s^2$ are related to other natural geometric theories. For instance, the Gauss map of constant mean curvature surfaces in $\r^3$ is harmonic for the conformal structure induced by the immersion \cite{Ru} (see also \cite{Ke}), whereas the Gauss map of positive constant Gaussian curvature is harmonic for the conformal structure of the second fundamental form \cite{GM}. The latter statement is the key in the proof of Theorem I-(ii). More precisely, we show that if the Gauss map of a surface of positive constant curvature in $\r^3$ is a diffeomorphism onto $\s^2-\{p\},$ then the conformal structure induced by the second fundamental form of the surface is that of $\c.$

On the other hand, using the harmonicity of the Gauss map of surfaces of positive constant curvature in $\r^3,$ harmonic diffeomorphisms from circular domains into domains in $\s^2$ bounded by a finite family of convex Jordan curves and satisfying a Neumann boundary condition have been recently shown by G\'{a}lvez, Hauswirth and Mira \cite{GHM}. It is an open question whether a harmonic diffeomorphism as those in Theorem I-(i) can be realized as the Gauss map of either a constant mean curvature or a constant Gaussian curvature surface in $\r^3.$

The paper is laid out as follows. In Section \ref{sec:preli} we state the necessary notations and preliminaries on harmonic maps between Riemannian manifolds and maximal graphs in Lorentzian product spaces. In Sections \ref{sec:thII} and \ref{sec:thI} we prove Theorems II and I, respectively. Also in Section \ref{sec:thI} we introduce some background on both Riemann surfaces and surfaces of positive constant curvature, for a well understanding of the proofs of items (i) and (ii) in Theorem I. Finally, in Section \ref{sec:maximal} we discuss about the relation between harmonic diffeomorphisms $U\to \s^2-\{p_1,\ldots,p_\mgot\}$ as those of Theorem I-(i) and conformal maximal immersions $U\to \s^2\times\r_1.$ 

%%%%%%
%%%%%%
%%%%%%

\section{Preliminaries}\label{sec:preli}

%\subsection{Harmonic maps between Riemannian manifolds}

Let $M=(M,g)$ and $N=(N,h)$ be smooth Riemannian manifolds. Given a smooth map $f:M\to N$ and a domain $\Omega\subset M$ with piecewise $C^1$ boundary $\partial \Omega,$  the quantity
\begin{equation}\label{eq:energyintegral}
E_\Omega(f)=\frac12\int_\Omega |df|^2 dV_g
\end{equation}
is called the {\em energy} of $f$ over $\Omega.$ Here $dV_g$ denotes the volume element of $M,$ and $|\cdot|$ the norm on $(N,h).$

A smooth map $f:M\to N$ is said to be {\em harmonic} if it is a critical point of the energy functional, that is, if for any relatively compact domain $\Omega\subset M$ and any smooth variation $F:M\times(-\epsilon,\epsilon)\to N$ of $f$ supported in $\Omega$ (i.e., $F$ is a smooth map, $f_0=f$ and $f_t|_{M-\Omega}=f|_{M-\Omega}$ $\forall t\in(-\epsilon,\epsilon),$ where $f_t:=F(\cdot,t):M\to N$ and $\epsilon>0$), the first variation $\frac{d}{dt} E_\Omega(f_t)|_{t=0}$ is zero.

If $M$ is $2$-dimensional, that is to say, a Riemannian surface, then the energy integral \eqref{eq:energyintegral} is invariant under conformal changes of the metric $g,$ hence so is the harmonicity of $f.$ Therefore, the harmonicity of a map from a Riemann surface to a Riemannian manifold is a well defined notion.
{ On the other hand, the harmonicity of a map is not preserved under conformal changes in the metric of the target manifold.}

See the surveys \cite{EL1,EL2,HW} for a good reference.

%\subsection{Maximal surfaces} 

\begin{remark}
Throughout the paper we denote by $\M=(\M,\langle\cdot,\cdot\rangle_\M)$ a compact $\ngot$-dimensional Riemannian manifold without boundary, $\ngot\in\n,$ $\ngot\geq 2.$
\end{remark}

We denote by $\M\times\r_1$ the Lorentzian product space $\M\times\r$ endowed with the Lorentzian metric
\[
\langle\cdot,\cdot\rangle=\pi^*_\M(\langle\cdot,\cdot\rangle_\M)-\pi^*_\r(dt^2),
\]
where $\pi_\M$ and $\pi_\r$ denote the projections from $\M\times\r$ onto each factor. 
For simplicity, we write
\[
\langle\cdot,\cdot\rangle=\langle\cdot,\cdot\rangle_\M-dt^2.
\]

A smooth immersion $X:\Sigma\to\M\times\R_1$ of a connected $\ngot$-dimensional manifold $\Sigma$ is said to be {\em spacelike} if $X$ induces a Riemanninan metric $X^*(\langle\cdot,\cdot\rangle)$ on $\Sigma.$ 
%If $X$ is spacelike then there exists a unique unitary timelike normal field $N$ globally defined in $\Sigma$ which is in the same time-orientation as %$\partial_t,$ so that
%\[
%\langle N,\partial_t \rangle\leq-1\quad \text{on }\Sigma.
%\]

Let $\Omega\subset \M$ be a connected domain and let $u:\Omega\to\r$ be a smooth function. Then the map 
\[
X^u:\Omega\to\M\times\r_1,\quad X^u(p)=(p,u(p))\;\forall p\in\Omega,
\]
determines a smooth graph over $\Omega$ in $\M\times\r_1.$ The metric induced on $\Omega$ by $\langle\cdot,\cdot\rangle$ via $X^u$ is given by 
\[
\langle\cdot,\cdot\rangle_u:=(X^u)^*(\langle\cdot,\cdot\rangle)=\langle\cdot,\cdot\rangle_\M-du^2,
\]
hence $X^u$ is spacelike if and only if $|\nabla u|<1$ on $\Omega,$ where $\nabla u$ denotes the gradient of $u$ in $\Omega$ and 
$|\nabla u|$ denotes its norm, both with respect to the metric $\langle\cdot,\cdot\rangle_\M$ in $\Omega.$ In this case the function $u$ is said to be spacelike as well. If $u$ is spacelike, then the mean curvature $H:\Omega\to\r$ of $X^u$ is given by the equation
\[
H=\frac1{\ngot}\, {\rm Div}\left(  \frac{\nabla u}{\sqrt{1-|\nabla u|^2}} \right),
\]
where ${\rm Div}$ denotes the divergence operator on $\Omega$ with respect to $\langle\cdot,\cdot\rangle_\M.$

A smooth function $u:\Omega\to\r$ and its graph $X^u:\Omega\to\M\times\R_1$ are said to be
{\em maximal} if $u$ is spacelike and $H$ vanishes identically on $\Omega.$

If $K\subset\M$ is compact then a function $u:K\to\r$ is said to be smooth (resp., spacelike, maximal) if and only if $u$ extends to an open domain containing $K$ as a smooth (resp., spacelike, maximal) function.

A locally Lipschitz  function $u:\Omega\to\r$ is said to be {\em weakly spacelike} if and only if $|\nabla u|\leq 1$ a.e. in $\Omega.$ In this case the graph $X^u$ is said to be weakly spacelike as well. If $u$ is weakly spacelike then the area of $X^u(\Omega)$ is given by
\begin{equation}\label{eq:area}
\Acal(u):=\int_\Omega \sqrt{1-|\nabla u|^2}\, dV_\M,
\end{equation}
where $dV_\M$ denotes the volume element of $\langle\cdot,\cdot\rangle_\M.$ A smooth $u:\overline{\Omega}\to\r$ is a critical point of \eqref{eq:area} if and only if $u$ is maximal. 

If $u:\Omega\to\r$ is maximal then $X^u:(\Omega,\langle\cdot,\cdot\rangle_u)\to (\M\times\r_1,\langle\cdot,\cdot\rangle)$ is a harmonic map. In particular  
\[
{\rm Id}:(\Omega,\langle\cdot,\cdot\rangle_u)\to (\Omega,\langle\cdot,\cdot\rangle_\M)
\]
is a harmonic diffeomorphism, and 
\[
u:(\Omega,\langle\cdot,\cdot\rangle_u)\to \r
\]
is a harmonic function.

%%%%%%
%%%%%%
%%%%%%

\section{Moduli space of maximal graphs with isolated singularities. Proof of Theorem II}\label{sec:thII}

Throughout this section let $\mgot\in\n$  and let $\Agot=\{(p_i,t_i)\}_{i=1}^{\mgot}$ denote a set of points in $\M\times\r$ such that $p_i\neq p_j$ if $i\neq j,$ $\forall i,j\in\{1,\ldots,\mgot\}.$ 

We denote by $\Mgot_\Agot$ the space of continuous functions $u:\M\to\r$ such that
\begin{itemize}
\item $u(p_i)=t_i$ $\forall i=1,\ldots,\mgot,$ and
\item $\hat{u}:=u|_{\M-\{p_i\}_{i=1}^{\mgot}}$ is maximal.
\end{itemize}

The following claim trivially follows from the maximum principle for maximal surfaces.
\begin{claim}\label{cl:1}
If $\mgot=1$ then $\Mgot_\Agot$ consists of the constant function $u\equiv t_1.$
\end{claim}

From now on in this section assume that $\mgot\geq 2.$ The set $\Agot$ is said to satisfy the spacelike condition if and only if
\[
|t_i-t_j|<{\rm dist}_\M(p_i,p_j)\quad \forall i,j\in\{1,\ldots,\mgot\},\; i\neq j,
\]
where ${\rm dist}_\M(\cdot,\cdot)$ means distance in $(\M,\langle\cdot,\cdot\rangle_\M).$

%In particular, one has the following
%\begin{corollary}\label{co:condi}
%If $\Agot$ does not satisfy the spacelike condition then $\Mgot_\Agot=\emptyset.$
%\end{corollary}

From now on assume that $\Agot$ satisfies the spacelike condition.

For each $(i,n)\in\{1,\ldots,\mgot\}\times\n$ consider an open disk $B_i^n$ in $\M$ satisfying that
$\partial B_i^n$ is a smooth Jordan curve,
 $\overline{B_i^n}\cap\overline{B_j^n}=\emptyset$ if $i\neq j,$
 $\overline{B_i^{n+1}}\subset B_i^n,$ and
$\{p_i\}=\cap_{n\in\n} B_i^n.$
Define $\Delta_n=\M-\cup_{i=1}^\mgot B_i^n,$ $n\in\n.$ Let $\{t_i^n\}_{n\in\n}$ be a sequence of real numbers converging to $t_i,$ $i=1,\ldots,\mgot.$ 

 Consider the function $\varphi_n: \partial\Delta_n\to\r$ such that

\begin{equation}\label{eq:dirichletdata}
\text{$\varphi_n|_{\partial B_i^n}=t_i^n, \, i=1,\ldots,\mgot.$}
\end{equation}

Since $\Agot$ is finite then $\Agot$ satisfies the spacelike condition if and only if there exists $\epsilon_\Agot\in(0,1)$ such that $|t_i-t_j|<\epsilon_\Agot\cdot {\rm dist}_\M(p_i,p_j)$ $\forall i\neq j.$ It follows that there exists $n_0\in\n$ such that for each $n\geq n_0,$ the function $\varphi_n$ is $\epsilon_{\Agot,n}$-Lipschitz for some $\epsilon_{\Agot,n} \in(0,1).$

It is proved in \cite[p. 202]{Fe} that  there exists an $\epsilon_{\Agot,n}$-Lipschitz extension 
 $\widetilde{\varphi}_n$  of $\varphi_n$ to $\Delta_n.$ More precisely, such an extension is given by the formula:

\begin{equation}\label{eq:extension}
\text{$\widetilde{\varphi}_n(p)= \inf \{ \varphi_n(x)+\epsilon_{\Agot,n}  {\rm dist}_\M(p,x),\,\, x\in {\partial \Delta_n}\}$,  for  $p\in \Delta_n$.}
 \end{equation}

 Notice that $\widetilde{\varphi}_n$ is smooth near $\partial\Delta_n$. A simple approximation argument then shows that 
 
\begin{claim}\label{cl:sp-cond}
For all $n\geq n_0$, there exists a smooth spacelike function $\overline\varphi_n:\Delta_n\to\r$ such that $\overline\varphi_n|_{\partial B_i^n}=t_i^n, \, i=1,\ldots,\mgot.$
\end{claim} 

Then by \cite[Theorem 5.1]{Ge}, there exists a maximal function $u_n:\Delta_n\to\r$ such that 
\begin{equation}\label{eq:dirichlet}
\text{$u_n|_{\partial B_i^n}=\overline\varphi_n|_{\partial B_i^n}=t_i^n,$ $\forall n\geq n_0.$}
\end{equation}

Notice that the sequence $\{u_n\}_{n\in\n}$ is uniformly bounded (the $u_n$'s are maximal and there are uniform bounds on the boundary of the $\Delta_n$'s). Moreover $|\nabla u_n|<1$ on $\Delta_n,$ hence Ascoli-Arzela's Theorem and a diagonal argument give that, up to passing to a subsequence, 

\begin{claim}\label{cl:converge}
$\{u_n\}_{n\in\n}$ uniformly converges on compact sets of $\M-\{p_i\}_{i=1}^\mgot=\cup_{n\in\n}\Delta_n$ to a weakly spacelike function $\hat{u}:\M-\{p_i\}_{i=1}^\mgot\to\r.$
\end{claim}

Moreover, from \eqref{eq:dirichlet}, the convergence of $\{t_i^n\}_{n\in\n}$ to $t_i$ and the Lipschitz continuity of $\hat{u}$ one has that

\begin{claim}\label{cl:points}
$\hat{u}$ extends to a weakly spacelike function $u:\M\to\r$ with $u(p_i)=t_i$ $\forall i=1,\ldots,\mgot.$
\end{claim}

By the results in \cite[\textsection6]{Ba}, 
%$\hat{u}$ maximizes locally the operator $\Acal$ under the corresponding volume constraint. Therefore \cite[Theorem 6.4]{Ba} gives that 
$\hat{u}$ is a maximal function except for a set of points $\Lambda\subset \M-\{p_i\}_{i=1}^\mgot$ given by
$\Lambda:=\big\{p\in\M-\{p_i\}_{i=1}^\mgot\;|\; (p,\hat{u}(p))=\gamma(s_0)$ for some $0<s_0<1,$
where $\gamma:[0,1]\to\M\times\r_1$ is a null geodesic such that $\gamma((0,1))\subset X^{\hat{u}}(\M-\{p_i\}_{i=1}^\mgot)$ and $\pi_\M(\{\gamma(0),\gamma(1)\})\subset \{p_i\}_{i=1}^\mgot\big\}.$

Since $\Agot$ satisfies the spacelike condition then $\Lambda=\emptyset$ and 

\begin{claim}\label{cl:existence}
$u\in\Mgot_\Agot.$
\end{claim}

Now let us show the following

\begin{claim}\label{cl:uni}
$\Mgot_\Agot$ consists of exactly one element.
\end{claim}
\begin{proof}
Consider $u_1,u_2\in\Mgot_\Agot.$ By compactness of $\M$ there exists  $r_1\in\r, \, r_1\geq 0,$ such that $r_1+u_1\geq u_2$ on $\M$ and the equality holds at a non-empty subset $W\subset\M.$ If $r_1>0$ then, since $u_1,u_2\in\Mgot_\Agot,$ we must have $W\cap(\M-\{p_i\}_{i=1}^\mgot)\neq\emptyset.$ By the maximum principle for maximal surfaces $r_1+u_1=u_2$ on $\M.$ This contradicts the fact that  $u_1,u_2\in\Mgot_\Agot.$ Therefore $r_1=0$ and so   $u_1\geq u_2.$ In a symmetric way we also have $u_2\geq u_1.$ Therefore $u_1=u_2$ and we are done.
\end{proof}

At this point notice that the first part of Theorem II in the introduction follows from Claims \ref{cl:existence} and \ref{cl:uni}. Even more,

\begin{remark}\label{re:CMC}
The first part of Theorem II holds with the same proof if in the sentence
\begin{itemize}
\item $\Sigma(\Agot)-\Agot$ is a spacelike maximal graph over $\M-\{p_i\}_{i=1,\ldots,\mgot}$
\end{itemize}
one changes ``maximal'' by ``constant mean curvature''.
\end{remark}

Write $\Mgot_\Agot=\{u_\Agot\}.$

Denote by $\Ggot_\mgot$ the set of functions $u_\Agot$ such that $\Agot$ consists of $\mgot$ points, $\Agot$ satisfies the spacelike condition and $|\nabla u_\Agot|=1$ at any point in $\pi_\M(\Agot).$

Consider a sequence $\{u_{\Agot_n}\}_{n\in\n\cup\{0\}}\subset\Ggot_\mgot.$ We say that $\{\Agot_n\}_{n\in\n}\to\Agot_0$ if and only if, up to a relabeling, $\{{\rm dist}_{\M\times\r}\big((p_i^n,t_i^n),(p_i^0,t_i^0)\big)\}_{n\in\n}\to 0$ $\forall i=1,\ldots,\mgot,$ where $\Agot_k=\{(p_i^k,t_i^k)\}_{i=1,\ldots,\mgot}$ $\forall k\in\n\cup\{0\},$ and  ${\rm dist}_{\M\times\r}(\cdot,\cdot)$ means distance in $\M\times\r$ with respect to the metric $\langle\cdot,\cdot\rangle_\M+dt^2.$ Likewise we define the convergence of a sequence of families of $\mgot$ points in $\M.$

\begin{claim}\label{cl:uniform}
$\{u_{\Agot_n}\}_{n\in\n}$ uniformly converges to $u_{\Agot_0}$ in the $C^0$ topology in $\M$ if and only if $\{\Agot_n\}_{n\in\n}\to \Agot_0.$
\end{claim}
\begin{proof}
Assume first that $\{u_{\Agot_n}\}_{n\in\n}$ uniformly converges to $u_{\Agot_0}$ in the $C^0$ topology in $\M.$ Up to passing to a subsequence, assume that $\{\pi_\M(\Agot_n)\}_{n\in\n}$ is convergent and let us show that $\{\pi_\M(\Agot_n)\}_{n\in\n}\to \pi_\M(\Agot_0).$ Indeed, suppose for a moment that there exist $p\in\pi_\M(\Agot_0)$ and an open geodesic disc $B$ centered at $p$ such that, up to passing to a subsequence, $\pi_\M(\Agot_n)\cap B=\emptyset$ for all $n\in\n.$ Reasoning as in the paragraph preceding Claim \ref{cl:existence}, Bartnik's results \cite{Ba} give that $u_{\Agot_0}$ is smooth and spacelike around $p,$ a contradiction. Then $\pi_\M(\Agot_0)\subset\lim_{n\to\infty}\pi_\M(\Agot_n).$ Since both sets consists of exactly $\mgot$ points then they agree. Since $\{u_{\Agot_n}\}_{n\in\n}\to u_{\Agot_0}$ and $\{\pi_\M(\Agot_n)\}_{n\in\n}\to \pi_\M(\Agot_0)$ then $\{\Agot_n\}_{n\in\n}\to \Agot_0$ as well.

For the converse assume that $\{\Agot_n\}_{n\in\n}\to \Agot_0.$ For each $(i,n,k)\in\{1,\ldots,\mgot\}\times\n\times\n$ consider an open disc $B_{i,n}^k$ in $\M$ such that $\partial B_{i,n}^k$ is a smooth Jordan curve, $\overline{B_{i,n}^k}\cap\overline{B_{j,n}^k}=\emptyset$ if $i\neq j,$ $\overline{B_{i,n}^{k+1}}\subset B_{i,n}^k,$ $p_i^n\in B_{i,n}^k$ and
\begin{equation}\label{eq:limit}
\text{for any compact $K\subset \M-\{p_1^0,\ldots,p_\mgot^0\}$ there exists $n_0\in\n$ s.t. $K\subset \Delta_n^n$ $\forall n\geq n_0,$}
\end{equation}
where $\Agot_n=\{(p_1^n,t_i^n),\ldots,(p_\mgot^n,t_\mgot^n)\}$ and $\Delta_n^k:=\M-\cup_{i=1}^\mgot B_{i,n}^k.$ Let $u_n^k:\Delta_n^k\to\r$ be a maximal function  satisfying
$u_n^k|_{\partial B_{i,n}^k}=t_i^n.$ (See the discussion preceding Claim \ref{cl:converge}.) By Claims \ref{cl:converge}, \ref{cl:points} and \ref{cl:existence} the sequence 
\begin{equation}\label{eq:limit2}
\text{$\{u_n^k\}_{k\in\n}$ uniformly converges in the $C^0$ topology on $\M$ to $u_{\Agot_n}.$}
\end{equation}
Taking into account \eqref{eq:limit}, a similar argument gives that the sequence 
\begin{equation}\label{eq:limit3}
\text{$\{u_n^{f(n)}\}_{n\in\n}$ uniformly converges in the $C^0$ topology on $\M$ to $u_{\Agot_0}$}
\end{equation}
as well, where $f:\n\to\n$ is any map with $f(n)\geq n$ $\forall n\in\n.$ Fix $p\in \M-\{p_1^0,\ldots,p_\mgot^0\}$ and $\epsilon>0.$ From \eqref{eq:limit2}, for any $n\in\n$ there exists $k_n\in\n$ such that
\begin{equation}\label{eq:limit4}
|u_n^k-u_{\Agot_n}|(p)<\epsilon/2\quad \forall k\geq k_n, 
\end{equation}
where we are assuming that $n$ and $k$ are large enough so that $p\in \Delta_n^k.$ Set $v_n:=u_n^h$ for $h:=\max\{k_n,n\}.$ Then \eqref{eq:limit3} gives $n_0\in\n$ such that
\begin{equation}\label{eq:limit5}
|v_n-u_{\Agot_0}|(p)<\epsilon/2\quad \forall n\geq n_0. 
\end{equation}
Combining \eqref{eq:limit4} and \eqref{eq:limit5} one has that $|u_{\Agot_n}-u_{\Agot_0}|(p)<\epsilon$ $\forall n\geq n_0.$ Since also  $\{\Agot_n\}_{n\in\n}\to \Agot_0,$ we conclude that $\{u_{\Agot_n}\}_{n\in\n}$ simply  converges to $u_{\Agot_0}.$ As $\M$ is compact and 
the $u_{\Agot_n}$ are weakly spacelike, this  convergence is uniform on $\M$ 
 and we are done.
\end{proof}

Consider $\mgot$ different points $\{p_1,\ldots,p_\mgot\}\subset\M$ and take $t_1=\ldots=t_{\mgot-1}\neq t_\mgot\in\r$ such that $\Agot:=\{(p_i,t_i)\}_{i=1,\ldots,\mgot}$ satisfies the spacelike condition. This is nothing but choosing $t_1$ and $t_\mgot$ close enough. By Claim \ref{cl:existence}, $u_{\Agot}$ is well defined and by the maximum principle for maximal surfaces, $u_{\Agot}\in\Ggot_\mgot.$ Hence,

\begin{claim}\label{cl:nonempty}
$\Ggot_\mgot\neq\emptyset$ for any $\mgot\geq 2.$
\end{claim}

Let $u\in \Ggot_\mgot.$ By definition, a mark in $u$ is an ordering $\Ogot=((q_1,r_1),\ldots,(q_\mgot,r_\mgot))$ of the points in $\Agot,$ where $u=u_{\Agot}.$ Then we say that $(u,\Ogot)$ is a marked function. We denote by $\Ggot_\mgot^*$ the space of marked functions in $\Ggot_\mgot.$ We define the maps
\[
\sgot_1:\Ggot_\mgot^*\to \Ggot_\mgot,\quad \sgot_1(u,\Ogot)=u,
\]
\[
\sgot_2:\Ggot_\mgot^*\to (\M\times\r)^\mgot,\quad \sgot_2(u,\Ogot)=\Ogot.
\]

By Claim \ref{cl:uni}, the map $\sgot_2$ is injective. Moreover,

\begin{claim}\label{cl:open}
$\sgot_2(\Ggot_\mgot^*)$ is an open subset of $(\M\times\r)^\mgot.$
\end{claim}
\begin{proof}
Let $(u,\Ogot)\in\Ggot_\mgot^*.$ Write $\Ogot=((q_1,r_1),\ldots,(q_\mgot,r_\mgot))$ and $\Agot=\{(q_1,r_1),\ldots,(q_\mgot,r_\mgot)\}.$ Since $u=u_{\Agot}$ then $\Agot$ satisfies the spacelike condition. Reason by contradiction and assume that there exists a sequence $\{\Ogot_n=((q_1^n,r_1^n),\ldots,(q_\mgot^n,r_\mgot^n))\}_{n\in\n}$ converging to $\Ogot$ in the metric topology of $(\M\times\r)^\mgot$ and $\Ogot_n\notin \sgot_2(\Ggot_\mgot^*)$ $\forall n\in\n.$ Write $\Agot_n=\{(q_1^n,r_1^n),\ldots,(q_\mgot^n,r_\mgot^n)\}$ and, up to passing to a subsequence, assume that $\Agot_n$ satisfies the spacelike condition $\forall n\in\n$ (recall that $\Agot$ does so). Write $u_n=u_{\Agot_n},$ $n\in\n.$ By Claim \ref{cl:uniform}, $\{u_n\}_{n\in\n}$ uniformly converges to $u$ in the $C^0$ topology on $\M.$ If, up to passing to a subsequence, $u_n:\M\to\r$ extends as a spacelike function to a point in $\pi_\M(\Agot_n),$ that can be assumed to be $q_1^n$ without loss of generality, $n\in\n,$ then again Bartnik's results \cite{Ba} give that $u$ extends as a spacelike function to $q_1$ as well, a contradiction. Then $u_n\in \Ggot_\mgot$ (recall that $\Agot_n$ consists of $\mgot$ points), hence $(u_n,\Ogot_n)\in \Ggot_\mgot^*$ and $\Ogot_n\in \sgot_2(\Ggot_\mgot^*),$ a contradiction. This proves the claim.
\end{proof}

We set $\overline{\Ggot}_{\mgot}= \sgot_2(\Ggot_\mgot^*).$ 
We can identify  $\Ggot_\mgot^*,$ endowed with the topology induced by the injection $\sgot_2$ into $(\M\times\r)^\mgot,$ 
with $\overline{\Ggot}_{\mgot}.$ The permutation group $\sigma_{\mgot}$ of order $\mgot ,$ acts naturally on $\Ggot_\mgot^*$ as follows:  for $\tau\in\sigma_{\mgot}$ and  $(u,\Ogot)\in\Ggot_\mgot^*$ with $\Ogot=((q_1,r_1),\ldots,(q_\mgot,r_\mgot))$,   we set $\tau . (u,\Ogot) = (u,\tau ( \Ogot))$ where 
$\tau ( \Ogot )= ((q_{\tau(1)},r_{\tau(1)}),\ldots,(q_{\tau(\mgot)},r_{\tau(\mgot)})).$ This action is clearly free and properly discontinuous and the orbit space is naturaly identified to ${\Ggot}_{\mgot}.$ By Claim 
\ref{cl:uniform} the topology induced by the covering map coincides with the topology of $C^0$-uniform convergence of graphs on $\M.$

This completes the proof of Theorem II.

%%%%%%%%%%%%%%%%%
%%%%%%%%%%%%%%%%%
%%%%%%%%%%%%%%%%%

\section{Existence or not of harmonic diffeomorphisms. Proof of Theorem I}\label{sec:thI}

Throughout this section we assume that $\M=(\M,\langle\cdot,\cdot\rangle_\M)$ is of dimension $\ngot=2,$ hence, a compact Riemannian surface without boundary.

Let us recall the following classification of Riemann surfaces. A compact Riemann surface (without boundary) is said to be elliptic. An open Riemann surface is said to be {\em hyperbolic} if it carries non-constant negative subharmonic functions, and it is said to be {\em parabolic} otherwise.
A Riemann surface $\Rcal$ with non-empty boundary is said to be {\em parabolic} if bounded harmonic functions on $\Rcal$ are determined by their boundary values. Otherwise, $\Rcal$ is said to be {\em hyperbolic}. (See \cite{ahlfors,Pe} for a good setting.) For instance, $\Rcal_1=\{z\in\c\;|\; 0<|z|\leq 1\}$ is parabolic whereas $\Rcal_2=\{z\in\c\;|\; \alpha<|z|\leq 1\},$ $\alpha\in(0,1),$ is hyperbolic.
\begin{remark}\label{re:parabolicity}
An open Riemann surface $\Rcal$ is parabolic if and only if $\Rcal-D$ is parabolic for any open relatively compact disc $D\subset\Rcal$ with smooth boundary.

Indeed, if $\Rcal$ is parabolic then, by \cite[$\mathsection$IV.3.3]{FK}, the Dirichlet problem has at most one bounded solution on $\Rcal-D,$ hence $\Rcal-D$ is parabolic as well. For the converse assume that $\Rcal$ is hyperbolic. Then, by \cite[$\mathsection$IV.3.4]{FK}, there exists a harmonic function $w$ on $\Rcal-D$ such that $0<w<1$ on $\Rcal-\overline{D}$ and $w=1$ on $\partial D,$ hence $\Rcal-D$ is hyperbolic and we are done.   
\end{remark}

Let $\mgot\in\n,$ $\mgot\geq 2,$ $u=u_\Agot\in \Ggot_\mgot$ and set $\Omega=\M-\pi_\M(\Agot).$ 

Recall that $X^u:(\Omega,\langle\cdot,\cdot\rangle_u)\to \M\times\r_1$ is a conformal harmonic map. Let $p\in\pi_\M(\Agot)$ and let $A$ be an annular end of $(\Omega,\langle\cdot,\cdot\rangle_u)$ corresponding to $p.$ Then $A$ is conformally equivalent to an annulus $A(r,1):=\{z\in\c\;|\; r<|z|\leq 1\}$ for some $0\leq r<1.$ Identify $A\equiv A(r,1)$ and notice that $u$ extends continuously to $S(r)=\{z\in\c\;|\;|z|=r\}$ with $u|_{S(r)}=u(p).$ By \cite{Ba2}, $X^u(A)$ is tangent to either the upper or the lower light cone at $X^u(p)$ in $\M\times\r_1.$ In particular $p$ is either a strict local minimum or a strict local maximum of $u.$ Then, up to a shrinking of $A,$ we can assume that $u|_{S(1)}$ is constant, where $S(1)=\{z\in\c\;|\;|z|=1\}.$ Since $u|_A$ is harmonic, bounded and non-constant then $r>0$ and $A$ has hyperbolic conformal type. This proves that

\begin{claim}\label{cl:hyperbolic}
$(\Omega,\langle\cdot,\cdot\rangle_u)$ is conformally an open Riemann surface with the same genus as $\M$ and $\mgot$ hyperbolic ends.
\end{claim}

In particular, one has the following

\begin{corollary}\label{co:diffeo}
Let $\M$ be a compact Riemannian surface, let $\mgot\geq 2$ and let $\{p_1,\ldots,p_\mgot\}\subset\M.$ Then there exist an open Riemann surface $\Rcal$ and a harmonic diffeomorphism $\phi:\Rcal\to\M-\{p_1,\ldots,p_\mgot\}$ such that every end of $\Rcal$ is of hyperbolic type.
\end{corollary}

By Koebe's uniformization theorem, any finitely connected planar domain is conformally equivalent to a domain in $\overline{\c}$ whose frontier consists of points and circles. In this setting the corollary above gives Item (i) in Theorem I, that is, one obtains the following existence result for harmonic diffeomorphism between hyperbolic and parabolic domains in $\s^2.$

\begin{corollary}\label{co:sphere}
Let $\mgot\in\n,$ $\mgot\geq 2,$ and let $\{p_1,\ldots,p_\mgot\}\subset\s^2.$ 

Then there exist a circular domain $U$ in $\overline{\C}$ and a harmonic diffeomorphism $\phi:U\to\s^2-\{p_1,\ldots,p_\mgot\}.$
\end{corollary}

{
Let us now show Theorem I-(ii).

The proof of Theorem I-(ii) fundamentally relies on the theory of surfaces of constant Gaussian curvature in Euclidean space. Before going into the details of the proof, let us state the necessary background on this theory.

Let $S$ be a smooth surface and let $X:S\to\r^3$ be an immersion with constant Gauss curvature $K$ equal to $1.$ For convenience we assume that $S$ is simply connected.

Up to changing orientation if necessary, the second fundamental form $II_X$ of $X$ is a positive definite metric. Therefore, $II_X$ induces on $S$ a conformal structure. Denote by $\Scal$ the Riemann surface with underlying differentiable structure $S$ and conformal structure induced by $II_X,$ and let $z=u+\imath v$ be a conformal parameter on $\Scal.$ Then $X$ may be understood as an immersion $X:\Scal\to\r^3$ and, following the results by G\'{a}lvez and Mart\'{i}nez \cite[$\mathsection$2.1]{GM}, the equation $K=1$ implies that 
\begin{equation}\label{eq:CGC}
X_u=N\times N_v \quad\text{and}\quad X_v=N\times N_u,
\end{equation}
where $N:\Scal\to\s^2$ denotes the unit normal vector field of $X.$ It follows that $N:\Scal\to\s^2$ is a harmonic local diffeomorphism.

Conversely, let $N:\Scal\to\s^2$ be a harmonic local diffeomorphism. Then the map $X:\Scal\to\r^3$ given by \eqref{eq:CGC} is an immersion with constant Gauss curvature $K=1$ (see \cite{GM} again and recall that $S$ is assumed to be simply connected).

On the other hand, in terms of the conformal parameter $z=u+\imath v,$ the first, second and third fundamental forms of $X:\Scal\to\r^3$ are given by
\begin{equation}\label{eq:ff}
\left\{
\begin{array}{rcc}
I_X=\langle dX,dX\rangle_{\r^3} & = & Qdz^2+2\mu|dz|^2+\overline{Q}d\overline{z}^2\\
II_X=\langle dX,dN\rangle_{\r^3} & = & 2\rho|dz|^2\\
III_X=\langle dN,dN\rangle_{\r^3} & = & -Qdz^2+2\mu|dz|^2-\overline{Q}d\overline{z}^2,
\end{array}
\right.
\end{equation}
where $\langle\cdot,\cdot\rangle_{\r^3}$ denotes the Euclidean metric in $\r^3,$ $Qdz^2$ is a holomorphic quadratic differential on $\Scal,$ and $\mu$ and $\rho$ are smooth positive real functions on $\Scal,$ see \cite{GHM}. Then, as Klotz pointed out in \cite[Remark 1]{Kl}, there exists an immersion $Y:\Scal\to\r^3$ achieving $III_X$ as its first fundamental form, $II_X$ as its second and $I_X$ as its third ones (recall that $S$ is simply connected and observe that $III_X$ is a positive definite metric). Since $X:\Scal\to\r^3$ is of constant Gauss curvature $K=1,$ it trivially follows from \eqref{eq:ff} that so is $Y:\Scal\to\r^3.$

Now we can prove Theorem I-(ii).
\begin{theorem}\label{pro:1}
There exists no harmonic diffeomorphism $\phi:\d\to\s^2-\{p\},$ $p\in\s^2.$ 
\end{theorem}
\begin{proof}
Let $\Scal$ be a simply connected Riemann surface and let $\varphi:\Scal\to\s^2-\{p\}$ be a harmonic diffeomorphism. To finish it suffices to check that $\Scal$ is conformally equivalent to the complex plane $\c.$

By \cite{GM}, since $\varphi:\Scal\to\s^2-\{p\}$ is a harmonic (local) diffeomorphism, then, up to replacing $\varphi$ by $-\varphi$ if necessary, there exists an immersion $X:\Scal\to\r^3$ with Gauss map $\varphi,$ constant curvature $K_X=1$ and such that 
the conformal structure of $\Scal$ is the one induced by the second fundamental form of $X,$ $II_X.$

Denote by $I_X$ and $III_X$ the first and third fundamental forms of $X,$ respectively. By \cite{Kl} there exists another immersion $Y:\Scal\to\r^3$ with constant curvature $K_Y=1,$ and such that
the first, second and third fundamental forms of $Y$ are given by $I_Y=III_X,$ $II_Y=II_X$ and $III_Y=I_X,$ respectively, and
\begin{equation}\label{eq:IIyIIx}
\text{the conformal structure of $\Scal$ is the one induced by $II_Y=II_X.$}
\end{equation}

Since $\varphi:\Scal\to\s^2-\{p\}$ is a diffeomorphism and $I_Y=III_X=\langle d\varphi,d\varphi\rangle_{\r^3}=\varphi^*(\langle \cdot,\cdot\rangle_{\s^2})$ (here $\langle\cdot,\cdot\rangle_{\s^2}$ denotes the canonical metric in $\s^2$), then $\varphi^{-1}:\s^2-\{p\}\to (\Scal,I_Y)$ is an isometry.
Since obviously $Y:(\Scal,I_Y)\to\r^3$ is an isometric immersion, then 
\[
\text{$Y\circ\varphi^{-1}:\s^2-\{p\}\to\r^3$ is an isometric immersion}
\]
as well. Following \cite[p. 419]{Po}, $Y\circ\varphi^{-1}$ is the restriction to $\s^2-\{p\}\subset\r^3$ of a rigid motion of $\r^3.$ In particular, $Y(\Scal)\subset\r^3$ is a once punctured round sphere. Therefore, the conformal structure induced on $\Scal$ by $II_Y=II_X$ is that of $\c.$ This and \eqref{eq:IIyIIx} conclude the proof. 
\end{proof}
}
\begin{remark}
Lemaire \cite{Le} showed that if a harmonic map $\varphi:\d\to N$ with finite energy satisfies that $\varphi|_{\s^1}$ is constant then $\varphi$ is constant as well, where $N$ is an arbitrary Riemannian manifold. The above theorem particularly shows that the condition on the energy of $\varphi$ can be removed if $\varphi$ is a diffeomorphism and $N=\s^2.$
\end{remark}

Finally Theorem I-(iii) is a very special instance of the following

\begin{proposition}\label{pro:converse}
Let $\Rcal$ be a parabolic open Riemann surface, let $N$ be an oriented Riemannian surface  and let $\phi:\Rcal\to N$ be a harmonic local diffeomorphism. Suppose either that $N$ has   Gaussian curvature  $K_N> 0$ or that  $K_N\geq 0$ and $N$ has no flat open subset.  

Then $\phi$ is either holomorphic or antiholomorphic.
\end{proposition}
\begin{proof}
Assume for instance that $\phi$ preserves orientation and let us check that $\phi$ is holomorphic.
Let $z$ (resp. $\phi$)  be a local conformal parameter in $\Rcal$ (resp. in $N$). The metric on $N$
writes $\rho(\phi) |d\phi|^2.$ A conformal metric on  $\Rcal$ writes $\lambda(z)|dz|^2.$  Following \cite{SY} we consider the following partial energy densities on  $\Rcal$:

\begin{equation} |\partial \phi|^2 = \frac{\rho(\phi(z)) }{\lambda(z)}\left| \frac{\partial \phi}{\partial z}\right|^2 \, \, 
 \text{and }\,\,   |\overline{\partial} \phi|^2 = \frac{\rho(\phi(z)) }{\lambda(z)}\left| \frac{\partial \phi}{{\partial} \bar z}\right|^2.
\end{equation}

Denote by $J(\phi)$ the Jacobian of $\phi.$ We have $J(\phi)= |\partial \phi|^2-  |\overline{\partial} \phi|^2.$ By our hypothesis $J(\phi)>0,$ that is, $ |\partial \phi|>  |\overline{\partial} \phi|.$ 

Reason by contradiction and assume that $\phi$ is not holomorphic, that is to say, $ |\overline{\partial} \phi| $ is not identically zero on $\Rcal.$ In this case, its zeroes are isolated \cite{SY}. Set $\Rcal^*:=\Rcal-\{|\overline{\partial} \phi| =0\}.$
We have 

\begin{equation}\label{eq:negative}
\log\frac {|\overline{\partial} \phi|}{|{\partial} \phi|}<0\quad\text{on $\Rcal^*.$}
\end{equation}

By the Bochner formula   (see again \cite[Chapter 1, \textsection7]{SY}):

%\[
%\Delta_{\Rcal} \log |{\partial} \phi| =-K_{N} J(\phi) + K_{\Rcal}
%\]
%\[
%\Delta_{\Rcal} \log |\overline{\partial} \phi| = K_{N} J(\phi) + K_{\Rcal}.
%\]
%Therefore 

\begin{equation}\label{eq:sub}
\Delta_{\Rcal}\log\frac{|\overline{\partial} \phi|}{|{\partial} \phi|}=2K_N J(\phi).
\end{equation}

Now note that the parabolicity of $\Rcal$ implies that of $\Rcal^*$ (see Claim \ref{cl:parabolicity} below). Suppose $K_N>0.$ By equations \eqref{eq:negative} and \eqref{eq:sub}, $\log\frac{|\overline{\partial} \phi|}{|{\partial} \phi|}$ is a  non-constant negative  subharmonic function on the parabolic surface $\Rcal^*,$ which is a contradiction.  Suppose now that $K_N\geq 0.$ Again, by the equations  \eqref{eq:negative} and \eqref{eq:sub}, the function $\log\frac{|\overline{\partial} \phi|}{|{\partial} \phi|}$ is subharmonic and hence constant. From  \eqref{eq:sub} we get $K_N J(\phi)\equiv  0.$ Since $J(\phi) > 0$ , we conclude that $K_N \equiv 0$ on the open set $\phi(\Rcal^*),$ which contradicts our hypothesis. 

In the case when $\phi$ reverses orientation then a parallel argument gives that $\phi$ is antiholomorphic. This concludes the proof.
\end{proof}

Since in the setting of Theorem I-(iii) the domains $\overline{\c}-\{z_1,\ldots,z_\mgot\}$ and $\s^2-\cup_{j=1}^\mgot D_j$ are not conformally equivalents, then the result holds.

For the lack of a reference, we now prove the following well known fact needed  in the proof of Proposition \ref{pro:converse}.

\begin{claim}\label{cl:parabolicity}
Let $\Rcal$ be an open parabolic Riemann surface and let $E\subset \Rcal$ be a closed subset consisting of isolated points.

Then $\Rcal^*:=\Rcal-E$ is an open parabolic Riemann surface.
\end{claim}
\begin{proof}
The fact that $\Rcal^*$ is an open Riemann surface is evident. Let us show that it is parabolic. Indeed, consider $B$ an open relatively compact disc in $\Rcal^*$ with smooth boundary and denote by $\Ncal$ the Riemann surface with boundary $\Ncal:=\Rcal^*-B.$ To finish it is suffices to prove that $\Ncal$ is parabolic (see Remark \ref{re:parabolicity}). Let $u:\Ncal\to\r$ be a non-constant bounded harmonic function with $u|_{\partial \Ncal}=0.$ Since $E$ consists of isolated points then $u$ extends harmonically to $\Ncal\cup E=\Rcal-B.$ Since $\Rcal-B$ is parabolic (see Remark \ref{re:parabolicity} again), $\partial (\Rcal-B)=\partial \Ncal$ and $u|_{\partial \Ncal}=0,$ then $u$ is identically zero on $\Ncal\subset \Rcal-B.$ This proves that $\Ncal$ is parabolic and we are done.
\end{proof}

This completes the proof of Theorem I. 
%%%%%%%%%%%%%%%%%
%%%%%%%%%%%%%%%%%
%%%%%%%%%%%%%%%%%

\section{Maximal graphs and harmonic diffeomorphisms between surfaces}\label{sec:maximal}

Let $\Rcal$ be a Riemann surface and let $N$ be a Riemannian surface. A map $X=(f,h):\Rcal\to N\times\R_1$ is conformal if and only if
\begin{equation}\label{eq:confo}
\left|\frac{\partial f}{\partial x}\right|^2 - \left(\frac{\partial h}{\partial x}\right)^2 = \left|\frac{\partial f}{\partial y}\right|^2 - \left(\frac{\partial h}{\partial y}\right)^2\quad\text{and}\quad
\left\langle \frac{\partial f}{\partial x}\,,\, \frac{\partial f}{\partial y} \right\rangle = \frac{\partial h}{\partial x} \frac{\partial h}{\partial y},
\end{equation}
where $z=x+\imath y$ is a local conformal parameter on $\Rcal$ and $|\cdot|$ and $\langle\,,\,\rangle$ denote the norm and metric on $N,$ repectively. If in addition $X$ is harmonic then the above equalities hold if and only if the Hopf differential of $f:\Rcal\to N,$
\[
\Phi_f= \left\langle \frac{\partial f}{\partial z}\,,\, \frac{\partial f}{\partial z} \right\rangle dz^2 = \frac14 \left( \left|\frac{\partial f}{\partial x}\right|^2 - \left|\frac{\partial f}{\partial y}\right|^2 -2\imath \left\langle\frac{\partial f}{\partial x}\,,\, \frac{\partial f}{\partial y} \right\rangle \right)dz^2,
\]
and the one of $h:\Rcal\to\r,$
\[
\Phi_h= \left(\frac{\partial h}{\partial z}\right)^2 dz^2 = \frac14 \left( \left(\frac{\partial h}{\partial x}\right)^2 - \left(\frac{\partial h}{\partial y}\right)^2 -2\imath \frac{\partial h}{\partial x}\frac{\partial h}{\partial y} \right)dz^2,
\]
agree.

Furthermore, a conformal harmonic immersion $X$ is spacelike (hence, a conformal maximal immersion) if and only if
\begin{equation}\label{eq:spa}
\left|\frac{\partial f}{\partial x}\right|^2 > \left(\frac{\partial h}{\partial x}\right)^2.
\end{equation}

On the other hand, let $\phi:\Rcal\to N$ be a harmonic map and denote by $\Phi_\phi$ its Hopf differential. Consider $(\widetilde{\Rcal},\Pi)$ a $2$-sheeted covering of $\Rcal$ such that $\widetilde{\Phi}_\phi:=\Phi_\phi\circ\Pi$ has a well defined square root, and write $\widetilde{\Phi}_\phi=(\widetilde{\varphi}(z)dz)^2$ on a local conformal parameter $z=x+\imath y$ on $\widetilde{\Rcal}.$ Observe that $(\widetilde{\Rcal},\Pi)$ is possibly branched  at the zeros of $\Phi_\phi.$

Consider now $(\widehat{\Rcal},\widehat{\Pi})$ a covering of $\widetilde{\Rcal}$ such that $\widehat{\varphi}:=\widetilde{\varphi}\circ\widehat{\Pi}$ has no real periods, and define
\[
X_\phi:\widehat{\Rcal}\to N\times\r_1,\quad X_\phi=(f_\phi, h_\phi),
\]
where 
\[
f_\phi:=\phi\circ\widetilde{\Pi}\circ\widehat{\Pi}\quad\text{and}\quad h_\phi:=\Re\int\widehat{\varphi}dz.
\]
Notice that $(\widehat{\Rcal},\widehat{\Pi})$ is infinitely sheeted unless $\widetilde{\varphi}$ has no real periods (recall that the periods are additive). 

Clearly, the Hopf differentials of $f_\phi$ and $h_\phi$ agree, so the above discussion gives that $X_\phi$ is a conformal harmonic map. 
Assume in addition that $\phi,$ and so $f_\phi,$ is a local harmonic diffeomorphism. From \eqref{eq:confo} and Cauchy-Schwarz inequality one has
\[
\left|\frac{\partial f_\phi}{\partial x}\right|\geq \left|\frac{\partial h_\phi}{\partial x}\right|.
\]
Assume the equality holds at a point $p\in\widehat{\Rcal}.$ Then \eqref{eq:confo} gives that, at the point $p,$ $|\partial f_\phi/\partial y|= |\partial h_\phi/\partial y|$ as well and $|\langle \partial f_\phi/\partial x\,,\, \partial f_\phi/\partial y\rangle| = |\partial f_\phi/\partial x|\cdot|\partial f_\phi/\partial y|.$ This contradicts that $f_\phi$ is a local diffeomorphism. Therefore $|\partial f_\phi/\partial x|>|\partial h_\phi/\partial x|$ on $\widehat{\Rcal}$ and $X_\phi:\widehat{\Rcal}\to\Ncal\times\r_1$ is a possibly branched conformal maximal immersion (see \eqref{eq:spa}).

In  this way we have showed the following

\begin{proposition}\label{pro:covering}
Let $\Rcal$ be a Riemann surface, let $N$ be Riemanninan surface and let $\phi:\Rcal\to N$ be a local harmonic diffeomorphism.

Then there exist a covering $(\widehat{\Rcal},\Pi)$ of $\Rcal$ and a possibly branched conformal maximal immersion $X_\phi=(f_\phi,h_\phi):\Rcal\to N\times\r_1$ such that $f_\phi=\phi\circ\Pi:\widehat{\Rcal}\to N.$ 
\end{proposition}

Let us now focus on the particular case when $U:=\Rcal\subset\c$ is a finitely connected circular domain, $N$ is the sphere $\s^2$ with a finite number of points removed and $\phi$ extends $C^1$ to the closure $\overline{U}$ of $U.$ Denote by $\Ncal$ the double of $\overline{U}$ (see \cite{St} for details on this construction). Recall that $\Ncal$ is a compact Riemann surface carrying an antiholomorphic involution $\Jcal:\Ncal\to\Ncal$ having the boundary of $\overline{U}$ as set of fixed points. Let $\Phi$ be a smooth quadratic differential on $\overline{U}$ and holomorphic on $U.$ Assume that $\Phi=\varphi(z)dz^2$ with $\varphi(x)\in\r$ $\forall x$ for any local conformal parameter $z=x+\imath y$ on $\overline{U}$ applying a piece of the boundary $\partial U$ of $\overline{U}$ into $\r\subset\c,$ then $\Phi$ extends holomorphically to $\Ncal$ in the form $\Jcal^\ast \Phi = \overline \Phi.$

Let $\phi:U\to N$ be a harmonic diffeomorphism extending $C^1$ to $\overline{U}.$ Obviously $\phi$ is constant over any connected component of $\partial U.$ Let $z=x+\imath y$ be a conformal parameter on $\overline{U}$ with $y|_{\partial U}=0.$ Then $\partial\phi/\partial x=0$ on $\partial U,$ hence the Hopf differential of $\phi$ can be written on $\partial U$ as 
\begin{equation}\label{eq:doble1}
(\Phi_\phi)|_{\partial U}=-\frac14\left|\frac{\partial\phi}{\partial y}\right|^2 dz^2.
\end{equation}
In particular, $\Phi_\phi$ extends holomorphically to $\Ncal$ with $\Jcal^{\ast}
\Phi_{\phi}=\overline{\Phi_{\phi}}.$ This particularly gives that 
\begin{equation}\label{eq:doble2}
\text{$\Phi_\phi$ has finitely many zeros on $U.$}
\end{equation}

Now, as above, we can take a $2$-sheeted covering $(\widetilde{U},\Pi)$ of $U$ such that $\widetilde{\Phi}_\phi:=\Phi_\phi\circ\Pi$ has a well defined square root. Write $\widetilde{\Phi}=(\varphi(z)dz)^2$ in a local conformal parameter $z$ on $\widetilde{U}.$ From \eqref{eq:doble1} one obtains that $\varphi(z)dz$ has no real periods. Then taking into account \eqref{eq:doble2} and following the discussion preceding Proposition \ref{pro:covering} one has the following

\begin{theorem}\label{th:covering}
Let $\phi:U\to\s^2-\{p_1,\ldots,p_\mgot\}$ be a harmonic diffeomorphism extending $C^1$ to $\overline{U},$ where $U$ is a finitely connected circular domain and $\{p_1,\ldots,p_{\mgot}\}$ is a finite subset in $\s^2.$ 

Then there exist a $2$-sheeted covering $(\widehat{U},\Pi)$ of $U$ and a possibly finitely branched conformal maximal immersion $X_\phi=(f_\phi,h_\phi):\widehat{U}\to\s^2\times\r_1$ such that $f_\phi=\phi\circ\Pi.$ 
\end{theorem}

In the proof of the above theorem, we have used that $\phi$ extends $C^1$ to $\overline{U}$ in order to obtain that the Hopf differential $\Phi_\phi$ of $\phi$ extends holomorphically to the double of $U.$ The authors do not know whether this hypothesis can be removed from the statement of the theorem.

%%%%%%%%%%%%%%%%%
%%%%%%%%%%%%%%%%%
%%%%%%%%%%%%%%%%%

\end{document}